\documentclass[letterpaper, 10 pt, conference, final]{ieeeconf}
\IEEEoverridecommandlockouts
\usepackage{amssymb}
\usepackage{amsmath}
\usepackage{breqn}
\usepackage{comment}
\usepackage{mathtools}
\usepackage{cite}
\usepackage{graphics}
\usepackage{algorithmic}

\usepackage{todonotes}

\DeclareMathOperator*{\argmin}{arg\,min}

\newtheorem{theorem}{Theorem}[section]
\newtheorem{definition}{Definition}[section]

\newtheorem{lemma}{Lemma}[section]
\newtheorem{proposition}{Proposition}[section]
\newtheorem{assumption}{Assumption}[section]
\newtheorem{remark}{Remark}

\numberwithin{equation}{section}

\begin{document}
%
\title{Multi-objective minimum time optimal control for low-thrust trajectory design}

\author{Nikolaus Vertovec
\and
Sina Ober-Bl{\"o}baum
\and
Kostas Margellos \thanks{Nikolaus Vertovec and Kostas Margellos are with the University of Oxford. Email:\{nikolaus.vertovec, kostas.margellos\}@eng.ox.ac.uk. Sina Ober-Bl{\"o}baum is with the University of Paderborn. Email:\{sinaober@math.uni-paderborn.de\}.}}

\maketitle

\begin{abstract}
We propose a reachability approach for infinite and finite horizon multi-objective optimization problems for low-thrust spacecraft trajectory design. The main advantage of the proposed method is that the Pareto front can be efficiently constructed from the zero level set of the solution to a Hamilton-Jacobi-Bellman equation. We demonstrate the proposed method by applying it to a low-thrust spacecraft trajectory design problem. By deriving the analytic expression for the Hamiltonian and the optimal control policy, we are able to efficiently compute the backward reachable set and reconstruct the optimal trajectories. Furthermore, we show that any reconstructed trajectory will be guaranteed to be weakly Pareto optimal. The proposed method can be used as a benchmark for future research of applying reachability analysis to low-thrust spacecraft trajectory design. 
\end{abstract}

%
\IEEEpeerreviewmaketitle

\section{Introduction}

Reachability analysis is an important research topic in the dynamics and control literature and has been used extensively for controller synthesis of complex systems \cite{Aubin2002, Lygeros1999}. In recent years we have also seen the use of reachability theory to design controllers that keep the state of the system in a "safe" part of the state space while steering the system towards a target set. Typically, these approaches rely on the computation of a capture basin (i.e. the set of points from which the target set can be safely reached within a given finite time). Computing such capture basins using a Hamilton-Jacobi-Bellman (HJB) approach has been shown in \cite{Margellos2011, Margellos2013, Bokanowski2017, Bokanowski2015, Fisac2015}. In \cite{Desilles2019} the authors propose an extension of the HJB approach to solve an infinite horizon multi-objective optimization problem (MOP) with state space constraints. Intuitively, in a multi-objective optimal control problem, one seeks to find the minimum control effort way a dynamical system can perform a certain task, while minimizing or maximizing a set of, usually contradictory and incommensurable, objective functions \cite{OberBlobaum2012}. A common example is found in spacecraft trajectory design, where the objective is to minimize the consumed propellant as well as transition time between two given orbits. However, since the final time is chosen as an optimization parameter, the approach described in \cite{Desilles2019} is no longer applicable. We, therefore, propose an extension that parameterizes the final time as an optimization variable by converting an infinite horizon control problem to a finite horizon one. The advantage of the proposed technique is that we are able to efficiently construct the Pareto front and optimal trajectories from the solution of a single HJB equation. This can make the comparison of multiple trajectories vastly more efficient compared with typical shooting methods \cite{Redding1984}.
This paper is organized into five sections. Section II contains details regarding the derivations of the spacecraft dynamics as well as the definitions of the constraints pertaining its behavior. In Section III the optimal control problem is formulated and the set of optimal trajectories is derived from the unique viscosity solution of a HJB equation. Section IV summarizes the simulation results obtained and discusses the numerical implementation. Finally, Section V provides concluding remarks and directions for future work.

\section{Modeling}
\subsection{Spacecraft equations of motion}
The spacecraft thrust can be modeled using the input $\textbf{u}(t) \coloneqq \begin{bmatrix} \textbf{u}_x(t), \textbf{u}_y(t), \textbf{u}_z(t) \end{bmatrix}^T \in \mathcal{U}$ where $\mathcal{U}$ is the set of possible control inputs. $\mathbf{u} \in \mathcal{U}_{ad}$ denotes the control policy and $\mathcal{U}_{ad}$ denotes the set of admissible policies which is the set of Lebesgue measurable functions from $[0, +\infty]$ to $\mathcal{U}$. Boldface notation is used to denote trajectories and non boldface notation is used to denote scalars and vectors.

The equations of motion of a particle or spacecraft around a rotating body can be expressed in 3-dimensional Euclidean space as a second-order ordinary differential equation \cite{Jiang2014}
\begin{multline}
	2 \mathbf{\Omega}(t) \times \frac{\partial \textbf{R}(t)}{\partial t}+\mathbf{\Omega}(t) \times (\mathbf{\Omega}(t) \times \textbf{R}(t))+\frac{\partial U(\textbf{R}(t))}{\partial \textbf{R}} \\
	+ \frac{\partial \mathbf{\Omega}(t)}{\partial t} \times \textbf{R}(t) -\frac{\textbf{u}(t)}{m(t)} = -\frac{\partial^2 \textbf{R}(t)}{\partial t^2},
	\label{eqn: system_dynamics}
\end{multline}
where \textbf{R}(t) is the radius vector from the asteroids center of mass to the particle, the first and second time derivatives of \textbf{R}(t) are with respect to the body-fixed coordinate system, $U(\textbf{R}(t))$ is the gravitational potential of the asteroid and $\Omega$ is the rotational angular velocity vector of the asteroid relative to inertial space. We consider an asteroid rotating uniformly with constant magnitude $\omega$ around the z-axis. Therefore, the Euler forces $\frac{\partial \mathbf{\Omega}(t)}{\partial t} \times \textbf{R}(t)$ can be neglected and we can express the rotation vector as $\Omega \coloneqq \omega \cdot e_z$, where $e_z$ is the unit vector along the z-axis. Following \cite{Greenwood1988}, the radius vector and its derivatives are given by 
\begin{equation}
	\mathbf{R}(t) \coloneqq \begin{bmatrix} \mathbf{x}(t) \\ \mathbf{y}(t) \\ \mathbf{z}(t) \end{bmatrix}, \quad 
	\frac{\partial \textbf{R}(t)}{\partial t} = \begin{bmatrix} \mathbf{v}_x(t) \\ \mathbf{v}_y(t) \\\mathbf{v}_z(t) \end{bmatrix}.
\end{equation}
The coriolis and centrifugal forces (the first two terms in \eqref{eqn: system_dynamics}) acting on the spacecraft are
\begin{equation}
	2  \Omega \times \frac{\partial \textbf{R}(t)}{\partial t} = \begin{bmatrix} -2 \omega \mathbf{v}_y(t)                         \\
			2 \omega \textbf{v}_x(t) \\
			0\end{bmatrix}, 
	\Omega \times (\Omega \times \textbf{R}(t)) =  \begin{bmatrix} -\omega ^2 \mathbf{x}(t) \\
			-\omega ^2 \mathbf{y}(t)                             \\
			0 \end{bmatrix}.
\end{equation}

Let us define the state vector $r \coloneqq \begin{bmatrix} x, y, z, v_x, v_y, v_z, m \end{bmatrix}^T \in \mathbb{R}^7$. Then following our derivations from \eqref{eqn: system_dynamics} we can formulate the system dynamics of the spacecraft as
\begin{multline}
	\dot{r}
	= \widetilde{f}(r, u) = \begin{bmatrix}
		v_x \\ v_y \\ v_z  \\
		U_x(x,y,z) + \omega ^2 x  + 2 \omega v_y  + \frac{u_x}{m}                \\
		U_y(x,y,z) + \omega ^2 y  - 2 \omega v_x + \frac{u_y}{m}      \\
		U_z(x,y,z) + \frac{u_z}{m} \\
		- \frac{\sqrt{u_x^2 + u_y^2 + u_z^2}}{v_{\mathrm{exhaust}}}\end{bmatrix},
	\label{eqn:dynamics}
\end{multline}
where $v_{\mathrm{exhaust}}$ is the exhaust velocity, $U_x$, $U_y$ and $U_z$ are the derivatives of the gravitational potential in the direction $e_x$, $e_y$ and $e_z$, respectively and where for brevity we neglect the time dependence by denoting $r=\textbf{r}(t)$, $v_x = \textbf{v}_x(t)$ etc.

\subsection{State constraints}
In order to ensure that the derived spacecraft dynamics hold, we need to enforce state constraints on $x, y, z$ as well as on the mass $m$.

Assuming that the burnout mass of the spacecraft is the same as the dry mass, then the total mass of the spacecraft is bounded by the amount of propellant available. We set $m_{\min} \coloneqq m_{\mathrm{dry}}$ and $m_{\max} \coloneqq m_{\mathrm{dry}} + m_{\mathrm{propellant}}$.
Since using all the propellant is never physically possible $m_{\min}$ is formulated as a strict inequality.

Due to particles ejected from the asteroid, we do not want to fall below a circular orbit with radius $\rho \coloneqq \sqrt{x^2 + y^2 + z^2}$ of approximately $\rho_{\min} = 1$ km. Furthermore, we need to stay within the sphere of influence (SOI) of the asteroid. The SOI can be approximated by $\rho_{SOI}  \approx a \left( \frac{M_1}{M_2}\right)^{\frac{2}{5}}$, where $a$ is the semi-major axis of the asteroid's orbit around the sun ($1.5907 \cdot 10^{8}$ km), $M_1$ is the Mass of the asteroid ($1.4091 \cdot 10^{12}$ kg) and $M_2$ is the mass of the sun ($1.9890 \cdot 10^{30}$ kg). Therefore, the sphere of influence of the asteroid is approximately $\rho_{\max} = 8.74$ km. Let us denote the set of states that satisfy the above assumptions as
\begin{equation*}
\mathcal{K}_0 \coloneqq \left\{ r \in \mathbb{R}^7 :  \rho \in [\rho_{\min}, \rho_{\max}], m \in (m_{\min}, m_{\max}]\right\},
\end{equation*} 
and let $\overline{\mathcal{K}_0}$ denote the closure of $\mathcal{K}_0$ and $\mathring{\mathcal{K}_0}$ the interior.

Whenever we approach the boundary of $\mathcal{K}_0$, we wish to be able to recover and reenter the interior $\mathring{\mathcal{K}_0}$. We, therefore, restrict ourselves to the set $\mathcal{K} \coloneqq \big\{ r \in \mathcal{K}_0 | \exists u \in \mathcal{U} : f(r,u) \cdot \eta_r < 0 \big\}$, where $\eta_r$ is the exterior normal vector to $\mathcal{K}$ at $r$. Recall that this need not hold for $m = m_{\min} \notin \mathcal{K}_0$. Overall, the set of state constraints we consider is encoded by the set $\mathcal{K}$, while the target orbit that we would like to transfer to lies within the nonempty closed target set $\mathcal{C} \subset \mathcal{K}$ defined as $\mathcal{C} \coloneqq \{r \in \mathcal{K} :  \quad  |r_{\mathrm{target}}-r| < \epsilon \}$, where $\epsilon > 0$ is an arbitrary tolerance.

\begin{remark}
	Notice that there is a connection between the set $\mathcal{K}$ and the viability kernel of $\mathcal{K}_0$. In fact, the viability kernel of $\mathcal{K}_0$ as defined in \cite{Aubin2011} will always satisfy $f(r,u) \cdot \eta_r < 0$.
\end{remark}

\section{Optimal Control problem}
The multi-objective optimal control problem can be formulated as a minimization problem using two objective functions in Mayer form. The first goal is to maximize the remaining mass, the second minimizes the required time for the orbit change, i.e. the terminal time. However, as the terminal time $t_f$ is unknown, we introduce a change of the time variable, i.e. for every $t_f \in [0,+\infty)$:
\begin{equation*}
	t_{t_f}(s) \coloneqq t_{0} + s \left(t_f - t_{0} \right) \quad \text{for} \quad s \in [0,1].
\end{equation*}

The new dynamics of the fixed final horizon problem (where the final horizon is 1 and $t_0 \coloneqq 0$) are then as follows:
\begin{multline}
	f(r,u,t_f) = (t_f - t_{0}) \widetilde{f}(r,u) \\
	\Rightarrow \begin{cases}
		\dot{\mathbf{r}}(s)= f(\mathbf{r}(s),\mathbf{u}(s),\boldsymbol\zeta(s)) & s \in [\kappa, 1] \\
		\dot{\boldsymbol\zeta}(s) = 0                                                     & s \in [\kappa, 1] \\
		\mathbf{r} ( \kappa ) = r_0                                                                  \\
		\boldsymbol\zeta(\kappa) = t_f
	\end{cases},
	\label{eqn: alt dynamics}
\end{multline}
where $\kappa$ is chosen from $[0,1]$ and $r_0$ is an initial state. The solution $\mathbf{r}$ belongs to the Sobolev space $\mathbb{W}^{1,1}([\kappa,1]; \mathbb{R}^7)$. The set of trajectory-control pairs on $[0,1]$ starting at $r_0$ with terminal time $t_f$ is denoted as:
\begin{multline*}
	\Pi_{r_0,t_f} \coloneqq \big\{(\mathbf{r},\mathbf{u}) : \dot{\mathbf{r}}(s) = f(\mathbf{r}(s),\mathbf{u}(s),\boldsymbol\zeta(s)), s \in [0,1]; \\
	\mathbf{r}(0) = r_0, \boldsymbol\zeta(0) = t_f\big\} \subset \mathbb{W}^{1,1}([0,1];\mathbb{R}^7) \times \mathcal{U}_{ad}.
\end{multline*}
Similarly, the set of admissible (in the sense of satisfying the state constraints) trajectory-control pairs on $[0,1]$ starting at $r_0$ with terminal time $t_f$ is denoted as:
\begin{multline*}
	\Pi_{r_0,t_f}^{\mathcal{K}, \mathcal{C}} \coloneqq \big\{(\mathbf{r},\mathbf{u}) \in \Pi_{r_0,t_f} : \mathbf{r}(s) \in \mathcal{K} \quad \text{for} \quad s \in [0,1]; \\
	\mathbf{r}(1) \in \mathring{\mathcal{C}} \big\} \subset \mathbb{W}^{1,1}([0,1]; \mathbb{R}^7) \times \mathcal{U}_{ad}.
\end{multline*}

Using Assumption \ref{assu:compact_convex} and \ref{assu:Lipschitz} that will be defined in Section \ref{subsection Aux}, as well as Filippov's Theorem \cite{Liberzon2011}, we can conclude, that $\Pi_{r_0,t_f}^{\mathcal{K}, \mathcal{C}}$ is compact.

Finally, the set of admissible terminal time and state pairs is denoted as 
\begin{multline*}
	\pi\coloneqq \big\{(r_0, t_f) \in \mathcal{K} \times [0,+\infty) \quad \text{such that} \quad \Pi_{r_0,t_f}^{\mathcal{K}, \mathcal{C}} \neq \emptyset \big\}.
\end{multline*}

For a given terminal state $r_f \in \mathbb{R}^7$ and terminal time $t_f \in [0, +\infty)$, we can define the costs functions as $J_1(r_f, t_f) \coloneqq -{r_f}_7$ and $J_2(r_f, t_f) \coloneqq t_f$, where ${r_f}_7$ denotes the 7th element of the state vector $r_f$ (the mass in our case). The 2-dimensional objective function $J : \mathbb{R}^7 \times [0,+\infty) \rightarrow \mathbb{R}^2$ can then be written as $J(r_f, t_f) \coloneqq \left[ J_1(r_f, t_f), J_2(r_f, t_f) \right]$. 

We are now in a position to formulate the multi-objective optimal control problem (MOC) under study by 
\begin{equation}
	\begin{cases}
		\inf J(\mathbf{r}(1), t_f) \\
		t_f \in [0, +\infty)                   \\
		(\mathbf{r}, \textbf{u}) \in \Pi_{r_0, t_f}^{\mathcal{K}, \mathcal{C}}.
	\end{cases}
	\label{eqn:MOC Problem}
\end{equation}

\subsection{Pareto Optimality} \label{subsection Pareto}
Before discussing how to solve \eqref{eqn:MOC Problem}, we will introduce two important concepts that are relevant when discussing multi-objective optimization. The first will be that of dominance between two admissible control pairs and the second will be weak and strong Pareto optimality \cite{Miettinen1998}.

\begin{definition}
	A vector $a$ is considered less than $b$ (denoted $a < b$) if for every element $a_i$ and $b_i$ the relation  $a_i < b_i$ holds. The relations $\leq, \geq, >$ are defined in an analogous way.
\end{definition}

\begin{definition}
	Let $(r_0, t_f), (r_0, \hat{t}_f) \in \pi$. We consider the trajectory-control pairs $(\mathbf{r}, \mathbf{u}) \in \Pi_{r_0, t_f}^{\mathcal{K}, \mathcal{C}}$ and $(\mathbf{x}, \mathbf{v}) \in \Pi_{r_0, \hat{t}_f}^{\mathcal{K}, \mathcal{C}}$.
	\begin{enumerate}
		\item The trajectory $\mathbf{r}$ dominates $\mathbf{x}$ if $J(\mathbf{r}(1), t_f) \leq J(\mathbf{x}(1), \hat{t}_f)$ and $J(\mathbf{r}(1), t_f) \neq J(\mathbf{x}(1), \hat{t}_f)$.
		\item The trajectory $\mathbf{r}$ strictly dominates $\mathbf{x}$ if $J(\mathbf{r}(1), t_f) < J(\mathbf{x}(1), \hat{t}_f)$.
	\end{enumerate}
\label{def:Dominance Relation}
\end{definition} 

\begin{definition}
	Let $(r_0, t_f) \in \pi$. We consider the trajectory-control pairs $(\mathbf{r}, \mathbf{u}) \in \Pi_{r_0, t_f}^{\mathcal{K}, \mathcal{C}}$.
	\begin{enumerate}
		\item The trajectory $\mathbf{r}$ is considered weakly Pareto optimal if $\forall \hat{t}_f \in [0, \infty),  \nexists (\mathbf{x}, \mathbf{v}) \in \Pi_{r_0, \hat{t}_f}^{\mathcal{K}, \mathcal{C}}$ such that $J(\mathbf{x}(1), \hat{t}_f) < J(\mathbf{r}(1), t_f)$.		
		\item The trajectory $\mathbf{r}$ is considered strictly Pareto optimal if $\forall \hat{t}_f \in [0, \infty), \nexists (\mathbf{x}, \mathbf{v}) \in \Pi_{r_0, \hat{t}_f}^{\mathcal{K}, \mathcal{C}}$ such that $J(\mathbf{x}(1), \hat{t}_f) \leq J(\mathbf{r}(1), t_f)$.
	\end{enumerate}
\label{def:Pareto Optimality}
\end{definition} 

Following Definition \ref{def:Pareto Optimality}, a trajectory $\textbf{r}$ is considered weakly Pareto optimal if it is not possible to improve all its performance metrics $J_1(\mathbf{r}(1), t_f), J_2(\mathbf{r}(1), t_f)$ simultaneously.
On the other hand, a trajectory $\textbf{r}$ is considered strongly Pareto optimal if no other admissible trajectory can ameliorate its performance metrics $J_1(\mathbf{r}(1), t_f), J_2(\mathbf{r}(1), t_f)$ without deteriorating the other. 

We now study how we can reconstruct Pareto optimal trajectories using a value function $\vartheta$. Letting $(\mathbf{r}, \textbf{u}) \in \Pi_{r_0, t_f}^{\mathcal{K}, \mathcal{C}}$ and $z \coloneqq \left( \begin{array}{c} z_1 \\ z_2 \end{array} \right) \in \mathbb{R}^2$, we now study how we can reconstruct Pareto optimal trajectories. First, we will choose a value function $\vartheta : \mathbb{R}^7 \times \mathbb{R}^2 \rightarrow \mathbb{R}$, discussed further in Section \ref{subsection Aux}, such that for any $t_f \in [0, \infty)$ the following holds:
\begin{multline}
	\vartheta(r_0, z)  \leq 0 \iff \Big[ \exists (\mathbf{r},\mathbf{u}) \in \Pi_{r_0, t_f}^{\mathcal{K}, \mathcal{C}},  J_1(\mathbf{r}(1), t_f) \leq z_1 \\
	\: \text{and} \: J_2(\mathbf{r}(1), t_f)  \leq z_2 \Big],
\end{multline}
\begin{equation}
		\forall z, z` \in \mathbb{R}^2, z \leq z` \Rightarrow \vartheta\left(r_0, z \right)  \geq \vartheta\left(r_0, z`\right).
\end{equation}

To discuss the following theorem, we introduce the utopian point $J^*(r_0)$, which is  the lower bound of $J$, defined elementwise as 
\begin{equation}
	J_i^*(r_0) \coloneqq \inf_{\{(\mathbf{r}(1), t_f) | \exists t_f : (r_0, t_f) \in \pi ; (\mathbf{r},\mathbf{u}) \in \Pi_{r_0, t_f}^{\mathcal{K}, \mathcal{C}} \}} J_i(\mathbf{r}(1), t_f).
\end{equation}
Furthermore, for the remainder of the paper, $a \bigvee b$ will denote $\max(a,b)$ and $\bigvee_i x_i$ will denote the maximum element of the vector $x$.
\begin{theorem}
	Let $z_1^*(r_0)$ and $z_2^*(r_0)$ be the two utopian values of $J_1$ and $J_2$ for a given initial state $r_0$ and let us define $z^*(r_0) \coloneqq \left(\begin{array}{c} z_1^*(r_0) \\ z_2^*(r_0) \end{array} \right)$ and $\mu \coloneqq \left( \begin{array}{c} \mu_1 \\ \mu_2 \end{array} \right) \in \mathbb{R}^2_{\geq 0}$. Moreover, let $\Pi_{r_0, z_2^*(r_0)}^{\mathcal{K}, \mathcal{C}} \neq \emptyset$. We consider the extended function $\Theta_{r_0}:[0,1]^2 \rightarrow[0,\infty]$ defined as follows:
	\begin{multline*}
		\Theta_{r_0}(\mu) \coloneqq \inf \Big\{ \tau \geq 0 \: \big| \: \vartheta(r_0,z^*(r_0)+ \mu\tau) \leq 0\\
		\text{and} \quad \tau < \frac{m_{\mathrm{propellant}}}{\mu_1}\Big\} \in [0,\infty ].
	\end{multline*}
	Let us define the 1-dimensional manifold:
	\begin{multline*}
		\Sigma_{r_0} \coloneqq \Big\{\left( \begin{array}{c} z_1^*(r_0) \\ z_2^*(r_0) \end{array} \right) + \mu \cdot \Theta_{r_0}(\mu ), \: \mu \coloneqq \left( \begin{array}{c} \mu_1 \\ \mu_2 \end{array} \right) \in [0,1]^2 \\
		\text{with} \: \mu_1 + \mu_2 = 1 \Big\}.
	\end{multline*}
	Any trajectory reconstructed from $\Sigma_{r_0}$, is guaranteed to be weakly Pareto optimal.
	\label{thrm:weakly pareto optimal}
\end{theorem}
\begin{proof}
	Let $z \in \Sigma_{r_0}$ and let us consider $\left( \mathbf{r}, \textbf{u} \right) \in \Pi_{r_0, t_f}^{\mathcal{K}, \mathcal{C}}$ such that $J(\mathbf{r}(1), t_f) \leq z$. By definition of $\Sigma_{r_0}$, there exists a $\mu \in [0,1]^2$ with $\sum_{k=1}^2 \mu_k = 1$ such that $z = z^*(r_0) + \mu  \Theta_{r_0}(\mu)$.

	First note, that if $\exists i \in \{1,2\}  : \mu_i = 0$, then $z_i = z_i^*(r_0)$ and $\nexists \left( \mathbf{x}, \mathbf{v} \right) \in \Pi_{r_0, t_f}^{\mathcal{K}, \mathcal{C}} : J(\mathbf{x}(1), t_f) < J(\mathbf{r}(1), t_f)$, as the utopian point by definition cannot be dominated. Therefore, we will assume that $\mu_i \neq 0$ for all $i \in \{1,2\} $ and show by contradiction that no pair $\left( \mathbf{x}, \mathbf{v} \right) \in \Pi_{r_0, t_f}^{\mathcal{K}, \mathcal{C}}$ can strictly dominate $\left( \mathbf{r}, \textbf{u} \right)$.

	Let us assume there exists a $\hat{t}_f \in [0, +\infty)$ and a pair $\left( \mathbf{x}, \mathbf{v} \right) \in \Pi_{r_0, \hat{t}_f}^{\mathcal{K}, \mathcal{C}}$ that strictly dominates $\left( \mathbf{r}, \mathbf{u} \right)$:
	$w\coloneqq J(\mathbf{x}(1), \hat{t}_f) < J(\mathbf{r}(1), t_f) \leq z$.
	Therefore, $ \vartheta(r_0, w) \leq 0$ and there exists a unique pair $(\nu, \varsigma) \in ]0,1]^2 \times \mathbb{R}_{>0}^2$ such that $\sum_{k=1}^2 \nu_k = 1$ and $w = z^*(r_0) + \nu \varsigma < z^*(r_0) + \mu  \Theta_{r_0}(\mu) \leq z$.

	If we assume that $\nu = \mu$, then consequently $w= z^*(r_0) + \mu \varsigma$ which implies that $\varsigma <  \Theta_{r_0}(\mu)$. However, since $ \vartheta(r_0, w) \leq 0$ it must hold that $\varsigma \geq  \Theta_{r_0}(\mu)$, a clear contradiction.

	Alternatively, we consider the case $\nu \neq \mu$. In this case we can define $\frac{\nu_{i_0}}{\mu_{i_0}} \coloneqq \bigvee_i \frac{\nu_i}{\mu_i} > 1$ for $i_0 \in \{1,2\} $. Introducing $\varsigma_0 \coloneqq  \frac{\nu_{i_0}}{\mu_{i_0}} \varsigma > \varsigma$ yields $w_{i_0} = z_{i_0}^*(r_0) + \nu_{i_0} \varsigma = z_{i_0}^*(r_0) + \mu_{i_0} \varsigma_0$. By definition of strict dominance, $w < z \iff \forall i \in \{1,2\}  \: w_i < z_i$ and therefore $w_{i_0} < z_{i_0} = z_{i_0}^*(r_0) + \mu_{i_0} \Theta_{r_0}(\mu)$. This implies that that $\varsigma < \varsigma_0 < \Theta_{r_0}(\mu)$, which is a contradiction as it violates $\varsigma \geq  \Theta_{r_0}(\mu)$. 

	In conclusion, there is no admissible pair $\left( \textbf{x}, \textbf{v} \right)$ that strictly dominates $\left( \mathbf{r}, \textbf{u} \right)$ and thus every trajectory reconstructed from $\Sigma_{r_0}$ is weakly Pareto optimal.
\end{proof}
Since weak Pareto optimality is of less relevance compared to strong Pareto optimality, we make the following observation.
\begin{lemma}
	The set of Pareto optimal values $z$, called the Pareto front $\mathcal{F}$, is a subset of $\Sigma_{r_0}$.
	\label{lemma: Pareto Front}
\end{lemma}
The proof of this Lemma is given in Appendix \ref{appendix: Pareto Front}. 

Following Lemma \ref{lemma: Pareto Front}, we can construct the Pareto front from $\Sigma_{r_0}$ by eliminating all points from $\Sigma_{r_0}$ that are dominated. For a discretized approximation of $\Sigma_{r_0}$, this can be done by iteratively removing any point $z$ that is dominated. Thus in conclusion, using the value function $\vartheta$, we are able to determine Pareto optimal solutions by constructing the set $\Sigma_{r_0}$ and simply eliminating dominated points.

\subsection{Auxiliary value function} \label{subsection Aux}
We now describe how the value function $\vartheta$ and the set $\Sigma_{r_0}$ can be computed. First, let $g(r)$ and $\nu (r)$ be two Lipschitz functions (with Lipschitz constants $L_g$ and $L_{\nu}$, respectively) chosen such that
\begin{align*}
	g(r) \leq 0 \iff r \in \overline{\mathcal{K}}, \\
	\nu (r) \leq 0 \iff r \in \ \mathcal{C}.
\end{align*}
This can be achieved by choosing $g(r)$ and $\nu(r)$ as signed distances to the sets $\mathcal{K}$ and $\mathcal{C}$, respectively. Then, by letting $(\mathbf{r}, \textbf{u}) \in \Pi_{r_0, t_f}^{\mathcal{K}, \mathcal{C}}$ and $z \coloneqq \left( \begin{array}{c} z_1 \\ z_2 \end{array} \right) \in \mathbb{R}^2$, we can describe the epigraph of the MOC problem, as defined in \eqref{eqn:MOC Problem}, by using the auxiliary value function $\omega$:
\begin{multline}
	\omega(\kappa, r_0, z, t_f) \coloneqq \inf_{(\mathbf{r}, \textbf{u}) \in \Pi_{r_0, t_f}} \Big\{
	\bigvee_i	J_i(\mathbf{r}(1), t_f) - z_i \\
	\bigvee
	\nu(\mathbf{r}(1))
	\bigvee
	\max_{s \in [\kappa, 1]} g(\mathbf{r}(s))
	\Big\}.
	\label{eqn: aux value func}
\end{multline}
Since, $\omega(\kappa, r_0, z, t_f) \leq 0$  implies that we have found an admissible trajectory that remains within $\overline{\mathcal{K}}$, by limiting $z_1$ to $z_1 <  -m_{\min} $ it follows that:
\begin{multline*}
	0 \geq \omega(\kappa, r_0, z, t_f) \geq J_1(\mathbf{r}(1), t_f) - z_1 = -m_{final} - z_1\\
	\iff m_{final} \geq - z_1 > m_{\min},
\end{multline*}
and thus the solution $\mathbf{r}$ lies within $\mathcal{K}$. 

For the following theorem, we need to make two assumptions about the spacecraft dynamics.

\begin{assumption}
	For every $r \in \overline{\mathcal{K}}$ the set $\Big\{\widetilde{f}(r,u) : u \in \mathcal{U} \Big\}$ is a compact convex subset of $\mathbb{R}^{7}$.
	\label{assu:compact_convex}
\end{assumption}

\begin{assumption}
	$\widetilde{f}:\mathcal{K} \times \mathcal{U} \rightarrow \mathbb{R}^{7}$ is bounded and there exists an $L_{f} > 0$ such that for every $ u_1, u_2 \in \mathcal{U}$,
	\begin{equation*}
		| \widetilde{f}(r_{1}, u_1) - \widetilde{f}(r_{2}, u_2) | \leq L_{f} |r_{1} - r_{2}|.
	\end{equation*}
	Moreover, following Assumption \ref{assu:compact_convex} there exists a $c_{f} > 0$ such that for any $r \in \overline{\mathcal{K}}$ we have $\max \big\{ |\widetilde{f}(r,u)| : u \in \mathcal{U} \big\} \leq c_{f}(1 + |r|)$.
	\label{assu:Lipschitz}
\end{assumption}

Under Assumption \ref{assu:Lipschitz} and following the Picard-Lindel{\"o}f theorem, for any control policy $\textbf{u} \in \mathcal{U}_{ad}$, any initial starting orbit $r \in \overline{\mathcal{K}}$ and terminal time $t_f \geq 0$, the system admits a unique, absolutely continuous solution on $[0,t_f]$ \cite{Dahmen2008}. By introducing the Hamiltonian $ H : \mathbb{R}^7 \times \mathbb{R} \times \mathbb{R}^7 \rightarrow \mathbb{R}$
\begin{equation*}
	H(r, t_f, q) \coloneqq \min_{u \in \mathcal{U}} \left(q^T \cdot f(r,u,t_f) \right),
\end{equation*}
we are able to now state how the auxiliary value function can be obtained.
\begin{theorem}
	 The auxiliary value function $\omega$ is the unique viscosity solution of the following HJB equation
	\begin{multline*}
	\begin{cases}
		\begin{aligned}
			\min(\partial_{\kappa} \omega + H(r, t_f, \nabla_{r} \omega), \omega(\kappa, r, z, t_f)  - g(r)) = {} & 0 \\
			\text{for} \quad \kappa \in [0,1),  \: r \in \mathcal{K}, z \in \mathbb{R}^2, t_f \in [0,+\infty),
		\end{aligned} \\\\
		\begin{aligned}
			\omega(1, r, z, t_f) = {} &
			\left(-r_7 -z_1 \bigvee t_f-z_2
			\bigvee
				\nu(r)
			\bigvee
				g(r)
			\right)   \\
			\text{for} \quad  r \in &\mathcal{K}, z \in \mathbb{R}^2, t_f \in [0,+\infty),
		\end{aligned}
	\end{cases}
	\end{multline*}
	\label{thm: HJB}
\end{theorem}
	Since the Dynamic Programming Principle holds for $0 \leq \kappa \leq \kappa + h \leq 1$, $r_0 \in \mathcal{K}$ and $z \in \mathbb{R}^2$ with $h \geq 0$:
	\begin{multline*}
		\omega(\kappa, r_0, z, t_f) = \inf_{(\mathbf{r}, \textbf{u}) \in \Pi_{r_0, t_f}} \Big\{
		\omega(\kappa + h, \mathbf{r}(\kappa +h), z, t_f) \\
		\bigvee \max_{s \in [\kappa, \kappa + h]} g(\mathbf{r}(s))
		\Big\},
	\end{multline*}
	the proof of Theorem \ref{thm: HJB} follows standard arguments for viscosity solutions, as shown in \cite{Altarovici2013, Margellos2011}. We will subsequently make some observations about the auxiliary value function $\omega$.

\begin{proposition}
	The auxiliary value function $\omega$ is Lipschitz continuous.
	\label{prop: omega-lips}
\end{proposition}
The proof of this Proposition is given in Appendix \ref{appendix: omega-lips}. 

\begin{proposition}
	Let $(\kappa, r_0, t_f) \in [0,1] \times \mathbb{R}^{7} \times [0, +\infty)$. The function $\omega(\kappa, r_0, z, t_f)$ has the following property:
	\begin{equation*}
		\forall z, z` \in \mathbb{R}^2,  z \leq z` \Rightarrow \omega\left(\kappa, r_0, z, t_f \right)  \geq \omega\left(\kappa, r_0, z`, t_f \right).
	\end{equation*}
	\label{prop: inverse-relation}
\end{proposition}
\begin{proof}
	Let $(\kappa, r_0, t_f) \in [0,1] \times \mathbb{R}^{7} \times [0, +\infty)$ and $z, z` \in \mathbb{R}^2$ with $z \leq z`$. Then for all $i \in {1,2}$, $J_i(\mathbf{r}(1), t_f) - z_i` \leq J_i(\mathbf{r}(1), t_f) - z_i,$, and consequently
	\begin{equation*}
		\left[ \bigvee_i J_i(\mathbf{r}(1), t_f) - z_i` \right] \leq \left[ \bigvee_i J_i(\mathbf{r}(1), t_f) - z_i \right].
	\end{equation*}
	Taking the maximum $\nu(\mathbf{r}(1)) \bigvee \max_{s \in [\kappa, 1]} g(\mathbf{r}(s))$ and the infimum over all $(\mathbf{r}, \textbf{u}) \in \Pi_{r_0, t_f}$, it follows from the last equation that $\omega\left(\kappa, r_0, z`, t_f \right) \leq \omega\left(\kappa, r_0, z, t_f \right)$.
\end{proof}
	
We now use the auxiliary value function $\omega$ to define $\vartheta$ and show that $\vartheta$ satisfies the requirements given in Section \ref{subsection Pareto}.
\begin{equation*}
	\vartheta(r_0, z) \coloneqq \min_{t_f}{\omega(0, r_0, z, t_f)}.
	\label{eqn: vartheta}
\end{equation*}

\begin{lemma}
Let $r_0 \in \mathbb{R}^{7}$. Then $\forall z, z` \in \mathbb{R}^2,  z \leq z` \Rightarrow \vartheta\left(r_0, z \right)  \geq \vartheta\left(r_0, z`\right)$.
	\label{thrm: inverse-relation}
\end{lemma}
\begin{proof}
	Following the definition of $\vartheta$, there exists $t_f, t_f' \in [0,+\infty)$ such that: $\vartheta(r_0, z)   = \omega(0, r_0, z, t_f)$ and $\vartheta(r_0, z`) = \omega(0, r_0, z`, t_f`)$. Subsequently we have:
	\begin{multline*}
		\vartheta(r_0, z') = \omega(0, r_0, z', t_f') \leq \omega(0, r_0, z', t_f) \\
		\leq \omega(0, r_0, z, t_f) = \vartheta(r_0, z).
	\end{multline*}
\end{proof}

Having shown how to construct $\vartheta$ from the solution of a HJB equation, we are now in the position to state and prove the following theorem, which is the main result of this section.

\begin{theorem}
	$\Sigma_{r_0}$ is defined by the zero level set of the value function $\vartheta$, i.e., $\Sigma_{r_0} = \left\{ z \in \mathbb{R}^2 \big| \vartheta(r_0, z) = 0 \right\}$.
	\label{thrm:zero-level-set}
\end{theorem}
\begin{proof}
	If $\Pi_{r_0, z_2^*(r_0)}^{\mathcal{K}, \mathcal{C}} \neq \emptyset$ then there must exist a pair $z \in \mathbb{R}^2_{\geq 0}$, such that $ \vartheta(r_0, z) \leq 0$. Since an admissible trajectory does not allow for all the propellant to be depleted, $z_1 < -	m_{\mathrm{dry}}$. The utopian point is $z_1^* \coloneqq -m_{\mathrm{dry}} - m_{\mathrm{propellant}} = -m_{\max}$. This is the case when no propellant has been burned.
	Thus, restricting $\tau$ to $\tau < \frac{m_{\mathrm{propellant}}}{\mu_1}$ in the definition of $\Theta_{r_0}$ does not exclude admissible trajectories from $\Sigma_{r_0}$, and subsequently $\Pi_{r_0, z_2^*(r_0)}^{\mathcal{K}, \mathcal{C}} \neq \emptyset \iff \Sigma_{r_0} \neq \emptyset$. Let us point out at this point, that since all admissible trajectories terminate in $\mathring{\mathcal{C}}$, $\nu(\mathbf{r}(1)) < 0$ holds for all $z \in \mathbb{R}^2$, and thus we can restrict ourselves to $\Pi_{r_0, t_f}^{\mathcal{K}, \mathcal{K}}$ without loss of generality.
	
	To show that $\Sigma_{r_0} = \left\{ z \in \mathbb{R}^2 \big| \vartheta(r_0, z) = 0 \right\}$, we need to show the following:
	\begin{align*}
		\Sigma_{r_0} \subseteq \left\{ z \in \mathbb{R}^2 \big| \vartheta(r_0, z) = 0 \right\} \quad \text{and}\\
		\left\{ z \in \mathbb{R}^2 \big| \vartheta(r_0, z) = 0 \right\} \subseteq \Sigma_{r_0}
	\end{align*}
	
	Let us assume that $ \Sigma_{r_0} \nsubseteq \left\{ z \in \mathbb{R}^2 \big| \vartheta(r_0, z) = 0 \right\} $ and therefore there exists a $z = z^*(r_0) + \mu \Theta_{r_0}(\mu) \in \Sigma_{r_0}$, such that $ \vartheta(r_0, z) < 0$. The case $ \vartheta(r_0,z) > 0$ can be excluded, as it follows from the definition of $\Theta_{r_0}$, that $\forall z \in \Sigma_{r_0} $, $ \vartheta(r_0, z) \leq 0$.

	$\vartheta(r_0, z) < 0$ implies that there exists a $t_f \in [0, +\infty)$ and $(\mathbf{r},\mathbf{u}) \in \Pi_{r_0, t_f}^{\mathcal{K}, \mathcal{K}}$ such that
	\begin{multline*}
		\vartheta(r_0,z^*(r_0)+ \mu \Theta_{r_0}(\mu)) = \\
		\bigvee_i J_i(\mathbf{r}(1), t_f) - z_i^*(r_0) - \mu_i \Theta_{r_0}(\mu)
		\bigvee
		\max_{s \in [0, 1]} g(\mathbf{r}(s))
		< 0 \\
		\Rightarrow \bigvee_i J_i(\mathbf{r}(1), t_f) - z_i^*(r_0) - \mu_i \Theta_{r_0}(\mu) < 0.
	\end{multline*}
	From the continuity of the value function, we can conclude that this implies the existence of a $\tau < \Theta_{r_0}(\mu)$, such that
	\begin{multline*}
		\bigvee_i J_i(\mathbf{r}(1), t_f) - z_i^*(r_0) - \mu_i \Theta_{r_0}(\mu) \\
		< \bigvee_i J_i(\mathbf{r}(1), t_f) - z_i^*(r_0) - \mu_i \tau \leq 0.
	\end{multline*}
	However, such a $\tau$ is contradictory to the definition $\Theta_{r_0}(\mu)$ and subsequently there is no $z$ in $\Sigma_{r_0}$ such that $ \vartheta(r_0, z)< 0$ and therefore $ \Sigma_{r_0} \subseteq \left\{ z \in \mathbb{R}^2 \big| \vartheta(r_0, z) = 0 \right\} $.

	To show $\left\{ z \in \mathbb{R}^2 \big| \vartheta(r_0, z) = 0 \right\} \subseteq \Sigma_{r_0}$, it suffices to show that if $z \notin  \Sigma_{r_0} \Rightarrow \vartheta(r_0, z) \neq 0$. Since by definition of $\Theta_{r_0} $, any $z$ smaller than $z^*(r_0)+ \mu \Theta_{r_0}(\mu)$ leads to a non-admissible trajectory (i.e $\vartheta(r_0, z) > 0$), we need only to show that the following relationship holds: 
	\begin{equation*}
		z > z^*(r_0)+ \mu \Theta_{r_0}(\mu) \Rightarrow \vartheta(r_0, z) < 0.
	\end{equation*}
	
	Let $\mu \in [0,1]^2$ with $\sum_k \mu_k = 1$ and let $\tau> \Theta_{r_0}(\mu)$. Then it follows from Theorem \ref{thrm: inverse-relation} and from the definition of $\Theta_{r_0}$, that $\vartheta(r_0,z^*(r_0)+ \mu\tau) \leq \vartheta(r_0,z^*(r_0)+ \mu \Theta_{r_0}(\mu)) = 0$.
	Therefore, there exists a $t_f \in [0, +\infty)$ and $(\mathbf{r},\mathbf{u}) \in \Pi_{r_0, t_f}^{\mathcal{K}, \mathcal{K}}$ such that
	\begin{multline*}
		\vartheta(r_0,z^*(r_0)+ \mu \Theta_{r_0}(\mu)) = \\
		\bigvee_i J_i(\mathbf{r}(1), t_f) - z_i^*(r_0) - \mu_i \Theta_{r_0}(\mu)
		\bigvee
		\max_{s \in [0, 1]} g(\mathbf{r}(s))
		= 0.
	\end{multline*}
	We now consider two cases.

	\textit{Case 1} Let us assume that $ \max_{s \in [0, 1]} g(\mathbf{r}(s)) < 0$.
	In this case, we have
	\begin{equation*}
		\bigvee_i J_i(\mathbf{r}(1), t_f) - z_i^*(r_0) - \mu_i \Theta_{r_0}(\mu)= 0.
	\end{equation*}
	Since $\tau > \Theta_{r_0}(\mu)$, we have $\bigvee_i J_i(\mathbf{r}(1), t_f) - z_i^*(r_0)- \mu_i \tau < \bigvee_i J_i(\mathbf{r}(1), t_f) - z_i^*(r_0)- \mu_i \Theta_{r_0}(\mu)=0$ and therefore:
	\begin{multline*}
		\vartheta(r_0,z^*(r_0)+ \mu\tau) \leq \\
		\left[
			\bigvee_i J_i(\mathbf{r}(1), t_f) - z_i^*(r_0)- \mu_i \tau
			\bigvee
			\max_{s \in [0, 1]} g(\mathbf{r}(s))
			\right] \\
		< \vartheta(r_0,z^*(r_0)+ \mu \Theta_{r_0}(\mu))  = 0.
	\end{multline*}

	\textit{Case 2} Let us now assume that $ \max_{s \in [0, 1]} g(\mathbf{r}(s)) = 0$.
	This implies, that the trajectory $\mathbf{r}$ touches the boundary $\mathcal{K}$ and $\bigvee_i J_i(\mathbf{r}(1), t_f) - z_i^*(r_0)- \mu_i \Theta_{r_0}(\mu) \leq 0$.

	We need to show that there exists $C > 0$ and $\varepsilon_0 > 0$ such that
	\begin{multline}
		\forall \varepsilon \in [0,\varepsilon_0), \:  \exists(\mathbf{r}_\varepsilon, \mathbf{u}_\varepsilon) \in \Pi_{r_0, t_f}^{\mathcal{K}, \mathcal{K}} \quad \text{such that} \\
		\quad \max_{s \in [0, 1]} g(\mathbf{r}_{\varepsilon}) \leq -\varepsilon \quad
		\text{and} \quad || \mathbf{r} - \mathbf{r}_{\varepsilon} ||_\infty \leq C \varepsilon.
		\label{eqn: g_eps}
	\end{multline}

	From Assumption \ref{assu:Lipschitz} we know that $|| \mathbf{r} - \mathbf{r}_{\varepsilon} ||_\infty$ is always bounded and therefore there exists a $C > 0$ for any trajectory $\mathbf{r}_{\varepsilon}$, such that $|| \mathbf{r} - \mathbf{r}_{\varepsilon} ||_\infty \leq C \varepsilon$.

	Let us now assume that $\nexists(\mathbf{r}_\varepsilon, \mathbf{u}_\varepsilon) \in \Pi_{r_0, t_f}^{\mathcal{K}, \mathcal{K}} \quad \text{such that} \quad \max_{s \in [0, 1]} g(\mathbf{r}_{\varepsilon}) \leq -\varepsilon$. This implies that all trajectories in $\Pi_{r_0, t_f}^{\mathcal{K}, \mathcal{K}} $ either leave or touch the boundary of $\mathcal{K}$. However, it follows from the definition of $\mathcal{K}$, that whenever we approach the boundary of $\mathcal{K}$, we can always find a $u$ such that we move away from $\partial \mathcal{K}$. Therefore, there must exists at least one trajectory that never leaves $\mathring{\mathcal{K}}$ and subsequently, equation \eqref{eqn: g_eps} holds.

	Using the Lipschitz continuity of $J$, we now obtain
	\begin{multline*}
		\bigvee_i J_i(\mathbf{r}_{\varepsilon}(1), t_f) - z_i^*(r_0) - \mu_i \tau \\
		\leq \bigvee_i J_i(\mathbf{r}(1), t_f) - z_i^*(r_0) - \mu_i \tau + L_J C \varepsilon.
	\end{multline*}

	By setting $\delta \coloneqq \max_{i} \mu_i(\tau - \Theta_{r_0}(\mu))$, we get
	\begin{multline*}
		\vartheta(r_0, z^*(r_0) + \mu \tau) \\
		\leq \bigvee_i J_i(\mathbf{r}_{\varepsilon}(1), t_f, )  - z_i^*(r_0) - \mu_i \tau
		\bigvee
		\max_{s \in [0, 1]} g(\mathbf{r}_{\varepsilon} (s)) \\
		\leq \vartheta(r_0, z^*(r_0) + \mu \Theta_{r_0}(\mu)) + (L_J C \varepsilon - \delta) \leq L_J C \varepsilon - \delta.
	\end{multline*}

	Now if we choose $\varepsilon < \frac{\delta}{L_J C}$, we can conclude
	\begin{multline*}
		\bigvee_i J_i(\mathbf{r}_{\varepsilon}(1),t_f)  - z_i^*(r_0) - \mu_i \tau
		\bigvee
		\max_{s \in [0, 1]} g(\mathbf{r}_{\varepsilon} (s)) < 0.
	\end{multline*}
	Therefore in all cases, we obtain $ \vartheta(r_0,z^*(r_0)+ \mu\tau) < 0$ and thus $\left\{ z \in \mathbb{R}^2 \big| \vartheta(r_0, z) = 0 \right\} \subseteq \Sigma_{r_0}$.
\end{proof}
\section{Numerical Approximation and Results}

\subsection{Numerical approximation of $\omega$}
Ultimately, the goal is to compute the set $\Sigma_{r_0}$ and the corresponding optimal trajectories. First, however, we will need to discuss how $\omega$ is computed from the HJB equation given in Theorem \ref{thm: HJB}. Since the term $J_2(\mathbf{r}(1), t_f)  - z_2$ does not depend on the system dynamics, we can omit it from the initial condition of the auxiliary value function $\omega$, and simply take the maximum of $\omega(0, r_0, z, t_f)$ and $J_2(\mathbf{r}(1), t_f)  - z_2$ to obtain the original $\omega(0, r_0, z, t_f)$, had $J_2(\mathbf{r}(1), t_f)  - z_2$ been included in the initial condition. This approach stems from an idea of system decomposition presented in \cite{Chen2018} and allows us, for the sake of solving the HJB equation, to omit one grid dimension, $z_2$. To solve the HJB equation we need to consider a uniform spaced grid $\mathcal{G} = \{r, z_1, t_f\}$ on $\mathcal{K} \times \mathbb{R} \times [0, +\infty)$, which enables us to use the Level Set Toolbox described in \cite{Mitchell2008} as well as some extensions presented in \cite{Bansal2018}.

To determine the set $\mathcal{K}$ it is possible to construct the viability kernel of $\mathcal{K}_0$ as shown in \cite{Saint-Pierre1994} and \cite{Frankowska1991}. However, numerical results have shown that by sufficiently constraining the considered grid points, we get an acceptable approximation of the set $\mathcal{K}$. Since $\mathcal{K}_0$ constrains the radius $\rho$, it is more efficient to compute the auxiliary value function $\omega$ in spherical coordinates. Let us define $\begin{bmatrix} a_\rho,  a_\theta, a_\psi \end{bmatrix}^T$ as the appropriate transformation of $a_x \coloneqq U_x(x,y,z) + \omega ^2 x  + 2 \omega v_y$, $a_y \coloneqq U_y(x,y,z) + \omega ^2 y  - 2 \omega v_x$ and $ a_z \coloneqq U_z(x,y,z)$. Then we can restate the system dynamics in spherical coordinates as
\begin{equation}
	\dot{r}  = \begin{bmatrix}
		v_\rho																	\\ 
		v_\theta 																\\ 
		v_{\psi} 																\\
		a_\rho  + \frac{T}{m} \cos{\alpha}                   			\\
		a_\theta+ \frac{T}{m} \sin{\alpha} \sin{\delta}        	\\
		a_\psi+ \frac{T}{m} \sin{\alpha} \cos{\delta} 			\\
		- \frac{T}{v_{\mathrm{exhaust}}}\end{bmatrix},
\end{equation}
where $v_\rho$, $v_\theta$ and $v_{\psi}$ are the velocities in the direction $e_r$, $e_{\theta}$ and $e_{\psi}$, respectively.
The input $\textbf{u}(t)$ is redefined for spherical coordinates as $\textbf{u}_{sph}(t) \coloneqq (\boldsymbol\alpha(t), \boldsymbol\delta(t), \textbf{T}(t)) \in \mathcal{U}$, where $\boldsymbol\alpha(t) \in  [-\pi, \pi]$ is the incidence angle, $\boldsymbol\delta(t) \in [-\frac{\pi}{2}, \frac{\pi}{2}]$ is the sideslip angle and $\textbf{T}(t) \in [0, T_{\max}]$ is the variable thrust.

To obtain a numerical approximation of $\omega$, we use a Lax-Friedrich Hamiltonian with an appropriate fifth-order weighted essentially non-oscillatory scheme as detailed in \cite{Osher2002}. As the PDE is solved backwards in time using a finite difference scheme, the Hamiltonian is given by $H(r, q) \coloneqq  -\min_{u \in \mathcal{U}} \left(f(r,u,t_f) \cdot q \right)$, where $q \in \mathbb{R}^7$ is the costate vector. Let us consider the term $C(r,q) \coloneqq
		q_0 v_\rho +
		q_1 v_\theta +
		q_2 v_\psi +
		q_3 a_\rho +
		q_4 a_\theta +
		q_5 a_\psi$.
Then we can write the Hamiltonian as follows
\begin{multline*}	
	H(r, t_f, q) \coloneqq -t_f \min_{u \in \mathcal{U}} \Big(
	\frac{T}{m} \big(
	q_3 \cdot \cos{\alpha} \\
	+ \sin{\alpha} \left(
		q_4 \cdot  \sin{\delta} +
		q_5 \cdot  \cos{\delta}
		\right) 
	\big)
	- q_6 \cdot  \frac{T}{v_{\mathrm{exhaust}}} 
	\Big) - t_f \cdot C(r,q).
\end{multline*}

As $T$ is always positive, we can minimize the term $\left(q_3 \cdot \cos{\alpha} +\sin{\alpha} \left(q_4 \cdot \sin{\delta} +q_5 \cdot  \cos{\delta}\right) \right)$ separately.
We can rewrite the trigonometric functions $a \cos{x} + b \sin{x}$ to $R \cos{(x - \arctan{\frac{b}{a}})}$ with $R = \sqrt{a^2 + b^2}$. 
We, therefore, introduce the auxiliary variables $\chi(\delta) \coloneqq \sqrt{q_4^2 + q_5^2} \cdot  \cos{(\delta - \arctan{\frac{q_4}{q_5}})}$ and $A(\delta) \coloneqq \sqrt{q_3^2 + \chi(\delta)^2}$.
We can first optimize over $\alpha$, and subsequently over $\delta$ (notice that this sequential minimization is exact since $A(\delta) \geq 0$)
\begin{multline*}
\min_{\alpha, \delta \in [-\pi, \pi]\times[-\frac{\pi}{2}, \frac{\pi}{2}]}\left(
	q_3 \cos{\alpha} +
	\sin{\alpha} \left(
		q_4 \sin{\delta} +
		q_5 \cos{\delta}
		\right) \right)  \\
	= \min_{\delta \in [-\frac{\pi}{2}, \frac{\pi}{2}]}  A(\delta) \cdot \min_{\alpha \in [-\pi, \pi]} \cos{(\alpha - \arctan{\frac{\chi(\delta)}{q_3}})}.
\end{multline*}
This results in the optimal thrust angles $\alpha^*(\delta) \coloneqq \pi + \arctan{\frac{\chi(\delta)}{q_3}}$.
Since $\cos{(\alpha^*(\delta) - \arctan{\frac{\chi(\delta)}{q_3}})} = -1$, after applying $\alpha^*(\delta)$, it follows that 
\begin{equation*}
	\delta^* \in \argmin_{\delta \in [-\frac{\pi}{2}, \frac{\pi}{2}]} - A(\delta) = \arctan{\frac{q_4}{q_5}} \pm \pi.
\end{equation*}

Subsequently, since $\sqrt{q_3^2 + q_4^2 + q_5^2} \geq A(\delta)$,
\begin{equation}
     q_3 \cos{\alpha^*} +
	\sin{\alpha^*} \left(
		q_4 \sin{\delta^*} +
		q_5 \cos{\delta^*}
		\right)
	= -\sqrt{q_3^2 + q_4^2 + q_5^2}.
	\label{eqn: angles}
\end{equation}
The results of \eqref{eqn: angles} allow us to minimize with respect $T$, i.e.
\begin{multline*}
	H(r, t_f, q) = -t_f \min_{T \in [0, T_{\max}]} \Big(
	-\frac{T}{m} \cdot \sqrt{q_3^2 + q_4^2 + q_5^2}\\
	- q_6 \cdot  \frac{T}{v_{\mathrm{exhaust}}} 
	\Big) - t_f \cdot C(r,q) \\
	\Rightarrow T^* \coloneqq \begin{cases}
		T_{\max}	 \quad \text{if} \: \frac{q_6}{v_{\mathrm{exhaust}}} + \frac{\sqrt{q_3^2 + q_4^2 + q_5^2}}{m} \geq  0\\
		0 	\quad \text{otherwise}
	\end{cases}.
\end{multline*}
Finally, applying $T^*$ and substituting $-\min(-\cdot)$ for $\max(\cdot)$, the Hamiltonian becomes
\begin{multline*}
	H(r, t_f, q) \coloneqq  - t_f \cdot C(r,q) + \\
	t_f \max{\Big(
		q_6 \cdot  \frac{T_{\max}	}{v_{\mathrm{exhaust}}} +
		\frac{T_{\max}	}{m} \sqrt{q_3^2 + q_4^2 + q_5^2}, 0
		\Big)}.
\end{multline*}
As in \cite{Bokanowski2015}, we can use this simplified expression of the Hamiltonian to achieve significant computational savings when computing $\omega$. For a discussion of the convergence of $\omega$ and the derivation of a necessary Courant-Friedrichs-Lewy condition, we refer to \cite{Bokanowski2010} and \cite{Mitchell2008}, respectively.

\subsection{Optimal trajectory reconstruction} \label{traj reconstruction}
The optimal control policy and trajectory $(\mathbf{r}, \mathbf{u}) \in \Pi_{r_0, t_f}^{\mathcal{K}, \mathcal{C}}$ can be constructed efficiently using the approximation of $\omega$ over $\mathcal{G}$. For a given $N \in \mathbb{N}$ we consider the timestep $h=\frac{1}{N}$ and a uniform grid with spacing $s^k = \frac{k}{N}$ of $[0,1]$. Let us define the state $(r^k)_{k=0,...,N}$ and control $(u^k)_{k=0,...,N-1}$ for the numerical approximation of the optimal trajectory and control policy. Setting $r_0$ as the initial orbit and choosing an appropriate $z$, we determine $\argmin_{t_f} \omega(s^0,r_0, z, t_f)$ to find a corresponding $t_f$. We then proceed by iteratively computing the control value
\begin{equation*}
	u^k(r^k) \in \argmin_{u \in \mathcal{U}}  \omega(s^k, r^k + h f(r^k,u,t_f), z, t_f) \bigvee g(r^k).
\end{equation*}
For a given $\omega(s^k, r^k, z, t_f)$ this is done by numerically taking the partial derivatives along each grid direction to estimate the costate vector $q$ and then determining the optimal control value as the minimizer of the Hamiltonian $H$. After $u^k$ is determined we compute $r^{k+1}$ using an appropriate Adams-Bashforth-Moulton method and increment $k$.

\subsection{Simulation results}
To illustrate the theoretical results of the previous sections, we consider a spacecraft on an unstable initial orbit around asteroid Castalia 4769.  The initial orbit spirals towards the asteroid and the spacecraft needs to make an orbit correction to a stable nominal target orbit so as to prevent a collision with the asteroid. The gravity of Castalia 4769 was modeled by means of a spherical harmonic expansion as discussed in \cite{Hudson1994, Scheeres1996}. 

For the numerical computation we considered the planar case, omitting the states $\psi$ and $v_\perp$. The spacecraft is modeled with $1000$ kg of dry mass, $100$ N of maximum thrust and an exhaust velocity of $40$ km/s.  Using a target orbit with radius $6.1175$ km and tangential velocity of $-0.0025$ km/s as well as $0.2$ kg of propellant, we are able to compute the numerical approximation of $\omega$. Following Theorem \ref{thrm:weakly pareto optimal} and Theorem \ref{thrm:zero-level-set}, any $z$ that satisfies $\min_{t_f}{\omega(0, r_0, z, t_f)} = \vartheta(r_0, z) = 0$ must belong to the set $\Sigma_{r_0}$. Using an initial orbit with radius $6.11$ km and tangential velocity of $-0.0026$ km/s,  we are able to compute $\Sigma_{r_0}$ by plotting the zero level set of $\vartheta(r_0, z)$, as shown in Fig. \ref{fig:sigma}. Taking an arbitrary $z$ from $\Sigma_{r_0}$ and finding the minimal optimal time $t_f \in \argmin_{t_f} \omega(0, r_0, z, t_f)$, we are able to construct the optimal trajectory as described in Section \ref{traj reconstruction} and shown in Fig. \ref{fig:Orbit_transition}. 

\begin{figure}[h]
	\centering
	\includegraphics[width=\columnwidth]{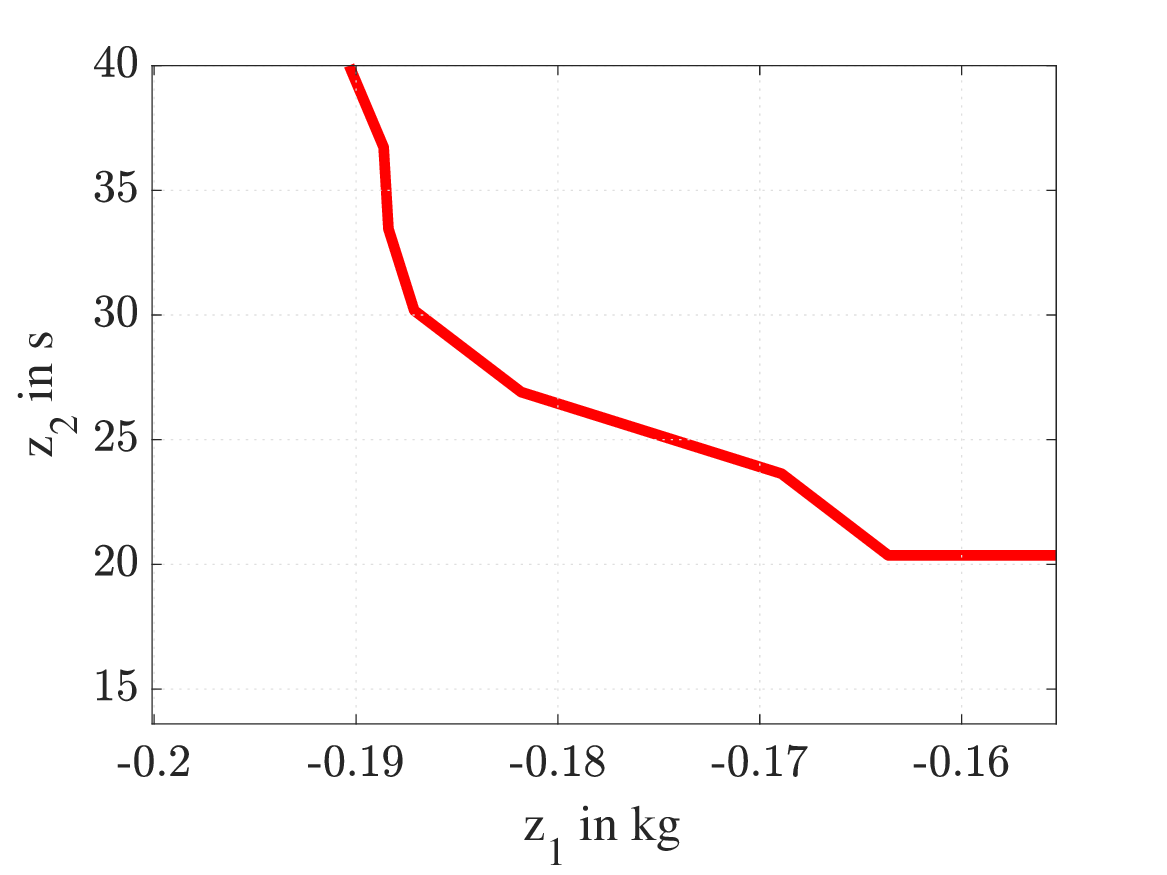}
	\caption{$\Sigma_{r_0}$ constructed from the zero level set of $\vartheta$.}
	\label{fig:sigma}
\end{figure}

\begin{figure}[h]
	\centering
	\includegraphics[width=\columnwidth]{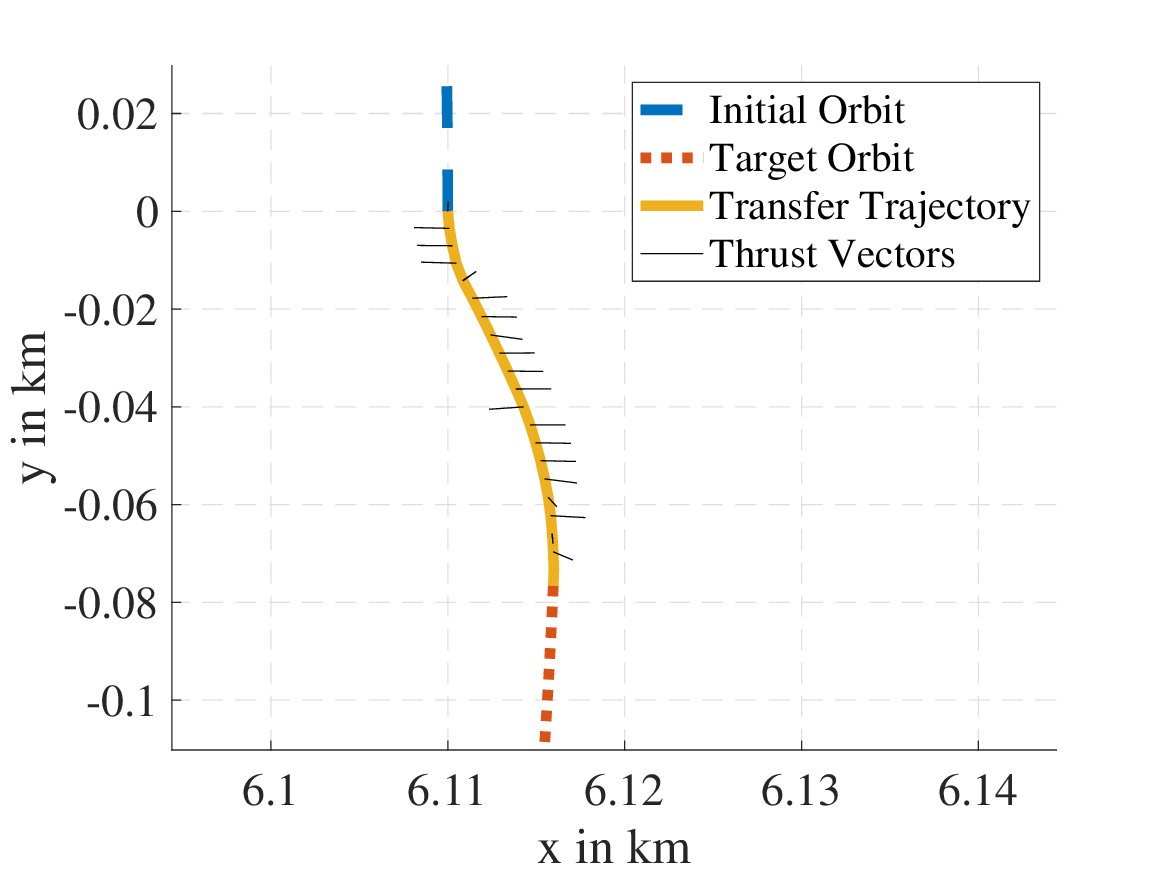}
	\caption{The initial orbit is given by the dashed line, the optimal transfer trajectory (corresponding to the z pair $[-0.18, 26.9]$) is reconstructed as described in \ref{traj reconstruction} and is given by the solid line, and the final orbit after the transition is complete is shown by the dotted line. The direction and magnitude of the continuous thrust is represented by the discretized set of solid lines shown along the transfer trajectory.}
	\label{fig:Orbit_transition}
\end{figure}

\section{Conclusion}
An approach to solving infinite and finite horizon multi-objective minimum time optimization problems was presented. Furthermore, we proved that strong and weak Pareto optimal values can be efficiently constructed from the zero level set of the unique viscosity solution of a Hamilton-Jacobi-Bellman equation. The feasibility and effectiveness of the proposed approach was demonstrated by applying it to the problem of low-thrust trajectory design. Future research concentrates towards constructing approximations of the reachable set and decomposing the optimization parameters efficiently, so as to allow higher accuracy and computational efficiency.


%

\appendices
\section{Proof of Lemma \ref{lemma: Pareto Front}} \label{appendix: Pareto Front}
\begin{proof}
	Let $z \in \mathcal{F}$ and let us consider a pair $(\mu, \tau) \in [0,1]^2 \times [0, \infty)$ with $\sum_{k=1}^2 \mu_k = 1$, such that $z = z^*(r_0) + \mu \tau$. Let us assume that $z \notin \Sigma_{r_0}$ and subsequently that $\tau > \Theta_{r_0}(\mu)$. This implies that there exists a $\hat{z} = z^*(r_0) + \mu \Theta_{r_0}(\mu) < z^*(r_0) + \mu \tau = z$ and consequently that $z$ is dominated by $\hat{z}$. Since by definition, no $z \in \mathcal{F}$ can be dominated, $\tau = \Theta_{r_0}(\mu)$ and therefore $z \in \Sigma_{r_0}$.
\end{proof}

\section{} \label{appendix: Lips-r}
\begin{proposition}
	Under Assumption \ref{assu:Lipschitz}, any trajectory $\mathbf{r}$ with terminal time $t_f$ reconstructed from $f$ is guaranteed to be Lipschitz continuous, with Lipschitz constant $L_{t_f} \coloneqq 1 + t_f L_f e^{t_f L_f}$.
	\label{prop:Lips-r}
\end{proposition}
\begin{proof}
Let $(r_0, t_f), (\hat{r}_0, \hat{t}_f) \in \mathbb{R}^7 \times [0, \infty)$ be two initial states and terminal times, then by Carath{\'e}odory's existence theorem, the following relation holds for arbitrary input policies $\mathbf{u}, \mathbf{\hat{u}} \in \mathcal{U}_{ad}$ and $t \in [0, 1]$:
\begin{multline*}
	\left| \mathbf{r}(t) - \mathbf{\hat{r}}(t) \right| \\
	\leq \left| r_0 - \hat{r}_0 \right| + t_f \int_0^t \left| \widetilde{f}(\mathbf{r}(s),\mathbf{u}(s)) - \widetilde{f}(\mathbf{\hat{r}}(s), \mathbf{\hat{u}}(s)) \right| ds \\
	\leq \left| r_0 - \hat{r}_0 \right| + t_f L_f \int_0^t \left| \mathbf{r}(s) - \mathbf{\hat{r}}(s)\right| ds.
\end{multline*}
Using the Bellman-Gronwall Lemma \cite{Sastry1999} it then follows that
\begin{multline*}
	\left| \mathbf{r}(t) - \mathbf{\hat{r}}(t) \right| \leq \left| r_0 - \hat{r}_0 \right| + \int_0^t \left| r_0 - \hat{r}_0 \right| t_f L_f e^{t \cdot t_f L_f }ds \\
	\leq \left| r_0 - \hat{r}_0 \right| (1 + t_f L_f e^{t_f L_f}) = L_{t_f} \left| r_0 - \hat{r}_0 \right|.
\end{multline*}
\end{proof}

\section{Proof of Proposition \ref{prop: omega-lips}} \label{appendix: omega-lips}
\begin{proof}
Let us fix $(r_0, z, t_f), (\hat{r}_0, \hat{z},\hat{t}_f)\in \mathbb{R}^7 \times \mathbb{R}^2 \times [0,\infty)$, $\kappa \in [0,1]$ and let $\epsilon > 0$. We choose $\hat{\mathbf{r}} \in \Pi_{\hat{r}_0,\hat{t}_f}$ such that
\begin{multline*}
	\omega(\kappa, \hat{r}_0, \hat{z},\hat{t}_f) \\
	\geq \bigvee_i	J_i(\mathbf{\hat{r}}(1), \hat{t}_f) - \hat{z}_i \bigvee \nu(\mathbf{\hat{r}}(1)) \bigvee \max_{s \in [\kappa, 1]} g(\mathbf{\hat{r}}(s)) - \epsilon.
\end{multline*}
By definition of $\omega$, for any $\mathbf{r} \in \Pi_{{r}_0,{t}_f}$, this yields the following relation
\begin{multline*}
	\omega(\kappa, r_0, z,t_f) - \omega(\kappa, \hat{r}_0, \hat{z},\hat{t}_f) \\
	\leq \bigvee_i	J_i(\mathbf{r}(1), t_f) - z_i \bigvee \nu(\mathbf{r}(1))
		\bigvee
		\max_{s \in [\kappa, 1]} g(\mathbf{r}(s)) \\
		- \bigvee_i	 J_i(\mathbf{\hat{r}}(1), \hat{t}_f) - \hat{z}_i \bigvee \nu(\mathbf{\hat{r}}(1)) \bigvee \max_{s \in [\kappa, 1]} g(\mathbf{\hat{r}}(s)) + \epsilon.
\end{multline*}
Let $\kappa_0 \in [\kappa, 1]$ be such that $g(\mathbf{r}(\kappa_0)) = \max_{s \in [\kappa, 1]} g(\mathbf{r}(s))$. Then subsequently $-\max_{s \in [\kappa, 1]} g(\mathbf{\hat{r}}(s)) \leq -g(\mathbf{\hat{r}}(\kappa_0))$ and
\begin{multline*}
	\omega(\kappa, r_0, z,t_f) - \omega(\kappa, \hat{r}_0, \hat{z},\hat{t}_f) \\
	\leq \bigvee_i	J_i(\mathbf{r}(1), t_f) - z_i \bigvee \nu(\mathbf{r}(1)) \bigvee g(\mathbf{r}(\kappa_0)) \\
		- \bigvee_i	J_i(\mathbf{\hat{r}}(1), \hat{t}_f) - \hat{z}_i \bigvee \nu(\mathbf{\hat{r}}(1)) \bigvee g(\mathbf{\hat{r}}(\kappa_0))+ \epsilon.
\end{multline*}
Using Proposition \ref{prop:Lips-r}, we define $L_r \coloneqq L_{t_f} \bigvee L_{\hat{t}_f}$ and show that in every case, there exists a Lipschitz constant. 

Case 1: $g(\mathbf{r}(\kappa_0)) \geq \omega(\kappa, r_0, z,t_f)$
\begin{multline*}
	\omega(\kappa, r_0, z,t_f) - \omega(\kappa, \hat{r}_0, \hat{z},\hat{t}_f) \\
	\leq g(\mathbf{r}(\kappa_0)) - \bigvee_i	 J_i(\mathbf{\hat{r}}(1), \hat{t}_f) - \hat{z}_i \bigvee \nu(\mathbf{\hat{r}}(1)) \bigvee g(\mathbf{\hat{r}}(\kappa_0))+ \epsilon \\
	\leq g(\mathbf{r}(\kappa_0)) - g(\mathbf{\hat{r}}(\kappa_0)) + \epsilon \leq L_g L_r \left| r_0 - \hat{r}_0 \right| + \epsilon
\end{multline*}
For the following cases the argumentation remains the same as in Case 1, and we simply state the final inequality.

Case 2: $\nu(\mathbf{r}(1)) \geq \omega(\kappa, r_0, z, t_f)$
\begin{multline*}
	\omega(\kappa, r_0, z, t_f) - \omega(\kappa, \hat{r}_0, \hat{z},\hat{t}_f)
	\leq L_{\nu} L_r \left| r_0 - \hat{r}_0 \right| + \epsilon
\end{multline*}

Case 3: $J_2(\mathbf{r}(1), t_f) - z_2 \geq \omega(\kappa, r_0, z, t_f)$
\begin{multline*}
	\omega(\kappa, r_0, z, t_f) - \omega(\kappa, \hat{r}_0, \hat{z},\hat{t}_f)
	\leq \left| t_f - \hat{t}_f \right| + \left| z_2 - \hat{z}_2 \right|  + \epsilon
\end{multline*}

Case 4: $J_1(\mathbf{r}(1), t_f) - z_1 \geq \omega(\kappa, r_0, z, t_f)$
\begin{multline*}
	\omega(\kappa, r_0, z, t_f) - \omega(\kappa, \hat{r}_0, \hat{z},\hat{t}_f)
	\leq \left| r_7 - \hat{r_7} \right| + \left| z_1 - \hat{z}_1 \right|  + \epsilon
\end{multline*}

Thus in every case there exists a set of constants $C_r, C_z$ and $C_{t_f}$ such that
\begin{multline*}
	\omega(\kappa, r_0, z, t_f) - \omega(\kappa, \hat{r}_0, \hat{z},\hat{t}_f) \\
	\leq C_r \left| r_0 - \hat{r}_0 \right| + C_z \left| z - \hat{z} \right| + C_{t_f} \left| t_f - \hat{t}_f \right| + \epsilon
\end{multline*}

The same argument conducted with $(r_0, z, t_f)$ and $(\hat{r}_0, \hat{z},\hat{t}_f)$ reversed establishes that 
\begin{multline*}
	\omega(\kappa, \hat{r}_0, \hat{z},\hat{t}_f) - \omega(\kappa, r_0, z, t_f) \\
	\leq C_r \left| r_0 - \hat{r}_0 \right| + C_z \left| z - \hat{z} \right| + C_{t_f} \left| t_f - \hat{t}_f \right| + \epsilon
\end{multline*}
Since $\epsilon$ is arbitrary, we conclude that 
\begin{multline*}
	\left| \omega(\kappa, \hat{r}_0, \hat{z},\hat{t}_f) - \omega(\kappa, r_0, z, t_f) \right| \\
	\leq C_r \left| r_0 - \hat{r}_0 \right| + C_z \left| z - \hat{z} \right| + C_{t_f} \left| t_f - \hat{t}_f \right|
\end{multline*}
\end{proof}
\section*{Acknowledgment}
We would like to thank Professor Bokanowski for correspondence on the numerical implementation. Furthermore, the authors would like to acknowledge the use of the University of Oxford Advanced Research Computing (ARC) facility in carrying out this work. http://dx.doi.org/10.5281/zenodo.22558



\bibliographystyle{./IEEEtran}
\bibliography{./IEEEabrv,./IEEEexample,./ECC}
%

\end{document}